\documentclass[12pt]{article}

\usepackage[margin=1in]{geometry}  % set the margins to 1in on all sides
\usepackage{graphicx}              % to include figures
\usepackage{amsmath}               % great math stuff
\usepackage{amsfonts}              % for blackboard bold, etc
\usepackage{amsthm}                % better theorem environments
\usepackage[utf8]{inputenc}
\usepackage[T1]{fontenc}
\usepackage{lmodern} % load a font with all the characters
\usepackage{latexsym}
\usepackage{hyperref}
\usepackage{color}

\newtheorem{thm}{Theorem}[section]
\newtheorem{lem}[thm]{Lemma}
\newtheorem{prop}[thm]{Proposition}
\newtheorem{cor}[thm]{Corollary}
\newtheorem{conj}[thm]{Conjecture}
\newtheorem*{remark}{Remark}

\newcommand\cE{{\mathcal E}}

\newcommand\cF{{\mathcal F}}

\newcommand\bR{{\mathbb R}}

\newcommand\bC{{\mathbb C}}

\newcommand\Aut{{\rm Aut}}

\newcommand\cO{{\mathcal O}}

\def\Im{\mathop{\rm Im}\nolimits}
\def\Re{\mathop{\rm Re}\nolimits}
\def\Ric{\mathop{\rm Ric}\nolimits}
\def\Vol{\mathop{\rm Vol}\nolimits}

\def\tr{\mathop{\rm tr}\nolimits}

\def\Aut{\mathop{\rm Aut}\nolimits}
\def\Iso{\mathop{\rm Iso}\nolimits}
\def\Ker{\mathop{\rm Ker}\nolimits}

\def\dbar{\overline\partial}
\def\ddbar{\partial\overline\partial}
\def\cO{{\mathcal O}}

\def\cE{{\mathcal E}}

\def\cF{{\mathcal F}}

\def\cL{{\mathcal L}}
\def\cC{{\mathcal C}}
\def\cH{{\mathcal H}}
\def\cD{{\mathcal D}}
\let\ol=\overline

\let\wt=\widetilde

\def\bC{{\mathbb C}}
\def\bR{{\mathbb R}}

\begin{document}

\title{On deformation of extremal metrics}
\date{\today}
\author{Xiuxiong Chen, Mihai P\u aun and Yu Zeng}

\maketitle

\section{Introduction}

In 1950's, E. Calabi (cf. \cite{Ca1}, \cite{Ca2}) proposed a program aiming to 
construct ``the best" metrics one could expect to find in a given K\"ahler class: these objects are 
currently called \emph{extremal metrics}.
To this end he has introduced a functional (the {Calabi energy}) so that  
the said metrics are obtained as 
critical points of it.
The K\"ahler-Einstein metrics (and more generally, the constant scalar curvature K\"ahler metrics
referred as cscK hereafter) are both special cases of extremal metrics. 

The main questions 
concerning the existence and uniqueness of K\"ahler-Einstein metrics on manifolds whose first Chern class is negative or zero
have been clarified 
in the 80's thanks to the fundamental contributions of Aubin, Calabi and Yau (cf. \cite{A}, \cite{calabi} and \cite{Y}, respectively). The remaining Fano case has been only recently settled by the crucial work of 
Chen-Donaldson-Sun (cf. \cite{CDS1}, \cite{CDS2}, \cite{CDS3}). 

After this major achievement, the study of extremal metrics should naturally be \emph{the} dominant subject
in the field. However, even the most
basic existence questions  concerning these metrics seem to be excessively difficult, given that the resulting partial differential equation
one has to deal with is of order four. Of course, the equation corresponding to a metric with prescribed Ricci curvature is of order four as well, but one can reduce it easily to a fully non-linear second order equation. This is no longer possible e.g. in the cscK case for general K\"ahler manifolds.

The continuity method is
a very powerful technique in PDE theory. It was successfully used by Aubin and Yau in 
their respective articles on K\"ahler-Einstein metrics. In \cite{CP}, the first named author 
proposed a continuity path which is very well adapted to the category of extremal metrics (regardless to their K\"ahler classes).
One can see that if all the geometric objects involved belong to a multiple of the canonical class,
then the path in \cite{CP} is obtained by taking the 
trace with respect to the solution metric of the continuity path used by Aubin and Yau. In this sense, it represents a natural extension of their techniques.
We refer to \cite{CP} for the proof of the basic facts about this new approach, including a crucial openness result and a few conjectural pictures.
\medskip

\noindent In the present article we are are pursuing this circle of ideas by establishing two deformation results about the cscK and extremal metrics, respectively. Let $(M, \omega)$ be a compact complex manifold endowed with a 
K\"ahler metric; we denote by $[\omega]\in H^{1,1}(X, \bR)$ the cohomology class corresponding to $\omega$.
We define the following space of potentials
\[
{\mathcal{H}^{\infty}(M)} = \big\{ \varphi \in C^{\infty}(M): \omega_\varphi = \omega+ \sqrt{-1} \partial \bar \partial \varphi > 0\big\}.
\]
If $\varphi\in \cH^\infty(M)$, then we denote by $R_\varphi$ the scalar curvature of the corresponding metric $\omega_\varphi$, and by $\underline R$ its average. i.e. 
$$\underline R:= \frac{1}{\Vol(M, \omega)}\int_MR_\varphi
\omega_\varphi^n.$$
Our first result states as follows.
\begin{thm}\label{thm1.1}
Let $(M, \omega)$ be a compact K\"ahler manifold such that there exists a cscK metric $\omega_{\varphi_0} \in [\omega]$. Then there exist $\epsilon > 0$ and a smooth function $\phi: (1-\epsilon, 1]\times M\to \bR$ such that 
$\varphi_t:= \phi(t, \cdot)\in \cH^\infty(M)$ and such that 
the corresponding metric verifies the equation
\begin{equation}\label{tcscK}
R_{\varphi_t} - \underline{R} - (1-t) (\tr_{\varphi_t}\omega - n) = 0.
\end{equation}
Moreover, there exists a holomorphic automorphism $f$ of $M$ such that 
$\omega_{\varphi_1}= f^\star \omega_{\varphi_0}$.
\end{thm}
\medskip

\noindent Following the terminology introduced by J. Fine \cite{JF} and J. Stoppa cf. \cite{JS}, a metric verifying the condition \eqref{tcscK}
is called \emph{twisted constant scalar curvature metric}. 
\medskip

\noindent The generalization of this notion in the context 
of extremal metrics was formulated in \cite{CP} as follows. 
A metric $\displaystyle \omega_\varphi$ is called \emph{twisted extremal K\"ahler metric} if
there exists $ t \in (0,1)$ such that the vector field
$$\nabla_{\varphi}^{1,0} \big(R_{\varphi}- (1-t) \tr_{\varphi} \omega\big)$$
is holomorphic. The result we obtain within this framework states as follows.

\begin{thm}\label{thm1.2}
Let $M$ be a compact complex manifold, and let $\omega$ be a K\"ahler metric on $X$, whose class $[\omega]$ contains an extremal K\"ahler metric $\omega_{\varphi_0}$. Then
there exists $\epsilon > 0$ together with a smooth function $\phi: ]1-\epsilon, 1]\times M\to \bR$ such that 
$\varphi_t:= \phi(t, \cdot)\in \cH^\infty(M)$ 
and such that the corresponding metric 
$\omega_{\varphi_t}$ is a twisted extremal metric, i.e.
$$\nabla_{\varphi_t}^{1,0} \big(R_{\varphi_t} - (1-t) \tr_{\varphi_t} \omega\big)$$
is a holomorphic vector field.
\end{thm}
\medskip

\noindent  As a direct consequence of Theorem \ref{thm1.2} we obtain a new proof of the following statement.

\begin{cor}\label{cor1.4}
Let $(M, \omega)$ be a compact K\"ahler manifold. Given two extremal metrics 
$\displaystyle (\omega_{j})_{j= 1, 2} \subset [\omega]$ there exists a holomorphic automorphism $f$ of $M$ such that $f^\star \omega_2= \omega_1$.
\end{cor}

\noindent The uniqueness problem above has long history which goes back to E. Calabi.  
Among the important articles generated by this question since then we refer to  
\cite{BM}, \cite{C}, \cite{CT}, \cite{D}, \cite{M2} as well as the recent paper \cite{BB}. 
To our knowledge, the fact that the original ideas of Bando-Mabuchi 
could be successfully used in order to establish the uniqueness of cscK/extremal metrics first appears in the paper of Berman-Berndsson cf. \cite{BB}; in addition they inject new convexity techniques in the field
(our arguments for the corollary above follow a similar approach).  

\medskip

\noindent For example, in the cscK case, the Corollary \ref{cor1.4} can be derived as follows.
We consider the one parameter family of twisted K-energy functional\footnote{For $s={1\over 2}$, this is already studied by a number of authors, \cite{JS}, \cite{MS} etc. }
\begin{align*}
\frac{\mathrm{d} E_{t}}{ \mathrm{d} s} = \int_M  \frac{\mathrm{d} \varphi}{\mathrm{d}s}(- t(R_{\varphi} -\underline{R}) + (1-t) (\tr_{\varphi} \omega - n)) \omega_{\varphi}^n,\qquad \forall \; t\in [0,1]
\end{align*}
and then we observe that the twisted cscK metrics are precisely the critical points of $E_t$
(up to a change of parameter). Next, we recall that thanks to \cite{C}, any two metrics in $[\omega]$ can be connected by a $C^{1,1}$ geodesic; on the other hand, if $t< 1$ then the functional $E_t$ above is \emph{strictly convex} along $C^{1,1}$ geodesics,
as a consequence of \cite{BB} and \cite{CLP}, together with the strict convexity of the $J$ functional
established in \cite{C2}. The convexity of the functional $E_1$ along smooth geodesics is due to T. Mabuchi;
the fact that this result still holds in the setting of $C^{1,1}$ geodesics is crucial for the proof.
By Theorem \ref{thm1.1}, we can deform the initial metrics $\omega_1$ and $\omega_2$ (modulo the action of a holomorphic automorphism of $M$) to twisted cscK metrics for which the 
corresponding parameter $t$ is strictly less than 1, so the corollary follows.

%plays important roles. This is a convex interpolation between the well known K-energy functional and the J functional which is introduced earlier in . The first named author conjectured earlier that the K-energy is convex along . This was recently proved by  Together with the convexity of $J_{\chi}$ functional (c.f. \cite{C2}), this easily implies that this one parameter family of twisted K-energy $E_{\chi, s}$ for $s \in (0,1)$ are all strictly convex over $C^{1,1}$ geodesic segment. It follows that twisted cscK metric is unique. By a similar argument on twising the modified K-energy\cite{FM} with J functional, we can also conclude that the twisted extremal metric is unique. Following the original bifurcation idea of \cite{BM}, this gives a new proof to the well known uniqueness theorem ((\cite{C}, \cite{CT},\cite{D}):
\medskip

\noindent In fact, the results 1.1 and 1.2 above represent the cscK version and the extremal version respectively of the Bando-Mabuchi work \cite{BM}.  
The proof we will present next is based on the \emph{bifurcation technique} developed in their celebrated
article \cite{BM} concerning the uniqueness up to biholomorphism of K\"ahler-Einstein metrics (see also Tian-Zhu \cite{TZ} for an analogous result in the context of K\"ahler-Ricci solitons). 
 Even if in our proof
(and in \cite{BB} likewise) one can easily recognize the main steps of 
the approach by Bando-Mabuchi, the techniques we had to develop/adapt in what follows are much more involved then the ones used in the K\"ahler-Einstein context. 
Moreover, the results above are important in their own right, because tightly connected with the program launched by E. Calabi.

%There have been tremendous developments about this topic over last few decades. With the final resolution of the K\"ahler-Einstein problem by ,  it is paramountly important to focus on the existence problem of the cscK metric which amounts to solve a 4th order PDE.  While there are already many fruitful approaches to this existing problem, in a recent article of X. Chen, the first named author proposed a new continuity path to attack the existence problem directly.  This path is amounting to solve a one parameter family of twisted cscK metrics, which links the critical point of the J functional to the desired cscK metric.  Let $(M,[\omega])$ be any polarized K\"ahler manifold without boundary. %where 
\medskip

\noindent Another motivation of the present article arise from the following conjecture, cf. \cite{CP}.
\begin{conj}\label{con}
For any $\chi > 0$ and $\chi \in [\omega]$, if $(M, [\omega], J)$ is destabilized by $(M, [\omega], J')$ where the later admits a cscK metric, then for any $s<1$ but sufficiently close to 1 there exists a twisted cscK metric for the triple structure $([\omega], \chi , s)$.
\end{conj}

\noindent This conjecture would give a new criteria for deciding whether a class $(M, [\omega], J)$ is 
semi-stable or not. If Conjecture \ref{con} holds, then given a semi-stable manifold, one would be able to find a continuous family of twisted cscK metrics for any $1-\varepsilon \leq t <1$. In principle, this should be sufficient for many geometric applications. Also, we remark that this is perfectly analog to the
classical K\"ahler-Einstein equation: along the continuity path, one can solve the Monge-Amp\`ere equation provided that 
the condition of K-semistability is satisfied. The conjecture above essentially states that the same phenomenon should occur in the context of the cscK metrics. Even though we cannot solve this conjecture yet, our main results here represent a slightly weaker existence result, which will hopefully lead to the solution of the conjecture itself in a near future.
\medskip

This article is organized as follows. In the first part we define a map between the Lie group of holomorphic automorphisms of $M$, and the space of normalized K\"ahler potentials. In the K\"ahler-Einstein setting,
the image of the differential of this map was computed by Bando-Mabuchi, cf. \cite{BM}, paragraph \S 6. 
Here we obtain an analog result, first for orbits corresponding to cscK 
metrics (cf. Proposition \ref{tangent}), and then in general, for orbits of extremal metrics in section 4. We remark 
that the case of extremal metrics is quite delicate, basically because of the fact that the Lichnerowicz operator (denoted by $\cD$ here) is not real in general. An important ingredient of the proof is a result due to Calabi, concerning the 
structure of the algebra of holomorphic vector fields on K\"ahler manifolds admitting an extremal metric.
We equally establish a Leibniz-type identity 
for the operator $\cD$; it is an elementary result, modulo the computations in the proof which are really involved. 
The complete arguments for our main results (i.e. the two theorems stated above) are given in sections three and four, respectively. We are using a version of 
the implicit function theorem, again modeled after \cite{BM}, but with many additional difficulties along the way. 
 
%%%%%%%%%%%%%%%%%%%%%%%%%%%%%%%%%%%%%%%%%%%%%%%%%%%%%%%%%%%%%%%%%%%%%%%%%%%%%%%%%%%%%%%%%%%%%%%%%%%%%%%%%%%%%%%%%%%%%%%%%%%%%%%%%%%%%%%%%%%%%%%%%%%%%%%%%%%%%%%%%%

\section{Preliminaries}\label{sect2}

Let $(M, \omega)$ be a compact K\"ahler manifold of dimension $n$; we denote by $[\omega]\in H^{1,1}(X, \bR)$ its corresponding cohomology class. We denote by $\Aut_0(M)$ the connected component of the Lie group of holomorphic automorphisms of $M$ containing the identity map, and let 
$\Iso_0(M, \omega)\subset \Aut_0(M)$ be the group of holomorphic isometries of $(M, \omega)$. 
\smallskip

\noindent It is well known that the quotient 
\begin{equation}\label{equa1}
\mathcal{O} = \Aut_0(M)/\Iso_0(M, \omega)
\end{equation}
is a homogeneous manifold. Let $\Gamma\subset H^0(M, T_M)$ be the vector space of holomorphic vector fields $X$ such that the Lie derivative $\cL_X\omega$ vanishes. Then the tangent space of $\cO$ is expressed as follows
\begin{equation}\label{equa2} 
T_{\cO}\simeq H^0(M, T_M)/\Gamma.
\end{equation}
\medskip

\noindent Along the next lines we will construct an embedding of $\cO$ into the space of potentials of the K\"ahler metric $\omega$, and we will identify the image of the corresponding tangent space.

Let $g\in \Aut_0(M)$ be a holomorphic automorphism of $M$. Then we have $g^\star\omega \in [\omega]$, so 
that there exists a real-valued function $\varphi\in \cC^\infty(M)$ such that
\begin{equation}\label{equa3} 
g^\star\omega= \omega_\varphi:= \omega+ \sqrt{-1}\ddbar \varphi.
\end{equation}
The function $\varphi$ is unique up to normalization; we introduce the normalized space of 
K\"ahler potentials
\begin{equation}\label{equa4} 
\wt \cH:= \{\varphi \in \cC^\infty(M) : \omega_\varphi> 0 \hbox{ and } \int_M \varphi\omega^n= 0\}.
\end{equation}
Then we have a well-defined map $\Psi^\omega: \Aut_0(M)\to \wt H$, such that $\Psi^\omega(g):= \varphi$ where the function 
$\varphi$ is uniquely defined by \eqref{equa3} together with the normalization in the definition of 
$\wt \cH$ in 
\eqref{equa4}. Moreover, we have $\Psi^\omega(g)= 0$ for any $g\in \Iso(M, \omega)$, and thus we obtain a map
\begin{equation}\label{equa5}
\Psi^\omega: \cO\to \wt H.
\end{equation}
\medskip

\noindent If the metric $\omega$ is cscK, then we can describe the image $\displaystyle \Psi^\omega_\star(T_{\cO, g})$
in a very simple manner, as follows.

\begin{prop}\label{tangent}
Let $(M, \omega)$ be a compact K\"ahler manifold, such that $\omega$ is cscK, and let $g\in \Aut_0(M)$ be an automorphism. Then the image of the tangent space $\displaystyle \Psi_\star^\omega(T_{\cO, g})$ coincides with 
the space generated by the real-valued functions $f\in \cC^\infty(M)$, such that $\nabla^{1,0}_\varphi f$ is holomorphic,
where $\varphi:= \Psi^\omega(g)$. In addition, the imaginary part of the vector field $\nabla^{1,0}_\varphi f$ is Killing with respect to the metric $\omega_\varphi$.
\end{prop}

\begin{proof}
To start with, let $g_t$ be a smooth path in $\cO$, such that 
$g_0= g$ and such that the derivative $\displaystyle \frac{dg_t}{dt}\big|_{t=0}$ identifies with a holomorphic vector field which we denote by $X$. There exists a smooth family $(\varphi_t)\subset \wt \cH$ with $\varphi_0 = \varphi$ so that we have 
\begin{equation}\label{equa6}
g_t^\star\omega:= \omega_{\varphi_t}
\end{equation}
for any parameter $t$. Since $\omega$ is cscK, so is $\displaystyle \omega_{\varphi_0}$. By \cite{Ca2}, we can decompose the holomorphic vector X as
\begin{align}
X = X_{a} + \nabla_{\varphi_0}^{1,0} ( f + \sqrt{-1}g),
\end{align}
where  $X_{a}$ is the autoparallel component of X, $f$ and $g$ are real-valued functions such that $\nabla_{\varphi_0}^{1,0} f$ and $\nabla_{\varphi_0}^{1,0}g$ are holomorphic (notice that here we are using the cscK condition)
We differentiate the relation \eqref{equa6} at $t=0$ and we obtain
\begin{equation}\label{equa7}
\sqrt{-1}\ddbar\dot\varphi_0 = \cL_{X_\bR}\omega_{\varphi_0} = \frac{1}{2} (\cL_{X} \omega_{\varphi_0} + \cL_{\bar X} \omega_{\varphi_0}) = \sqrt{-1} \partial \bar \partial f.
\end{equation}

\noindent  Thus, we know that $\dot \varphi_0 = f - \int f \omega^n$. The fact that the imaginary part of the holomorphic vector field 
$\displaystyle \nabla^{1,0}_{\varphi_0}\dot\varphi_0$ is Killing can be seen as a consequence of the fact that the function $\dot\varphi_0$ is real-valued, so we will not detail this point any further.
\end{proof}
\medskip

The next paragraph of this section is crucial: among the $\omega$--potentials belonging to the image 
$\Psi^\omega(\cO)$, we have to choose one which will enable us later to use the implicit function theorem in the proof of our main results. This is completely analogue to the paragraph \S 6 in \cite{BM}.

For any positive $(1,1)$-form $\chi$,  the $J_\chi$ functional introduced in \cite{C2} is defined as follows
\[
  {{\mathrm{d}  J_\chi}\over {\mathrm{d} t}} =  \int_M \tr_\varphi \chi  {{\mathrm{d} \varphi}\over {\mathrm{d} t}}  \frac{\omega_\varphi^n}{n!},\qquad \forall \varphi \in \mathcal{H}^{\infty}(M).
\]

We consider the functional $\iota:=  J_\omega- nI$, and by a direct computation, we obtain

\begin{equation}\label{grad}
\frac{\mathrm{d}}{\mathrm{d} t} \iota(\varphi_t) = \int_M (\text{tr}_{\varphi_t} \omega - n) \dot{\varphi_t} \frac{\omega_{\varphi_t}^n}{n!}.
\end{equation}
as well as 

\begin{align*}
\frac{\mathrm{d}^2 }{\mathrm{d}t^2} \iota(\varphi_t)= \int ( \ddot{\varphi} - |\nabla \dot{\varphi}|_{\varphi_t}^2) (\text{tr}_{\varphi_t}\omega -n) \omega_{\varphi_t}^n + \int \dot{\varphi}_{,\alpha}\dot{\varphi}_{,\bar{\beta}} \omega_{\bar{\alpha} \beta} \omega_{\varphi_t}^n >0.
\end{align*} 
so that in particular the functional $\iota$ is strictly convex along smooth geodesics.
\smallskip

\noindent As a consequence, we infer the following result, corresponding to \cite{BM}, Lemma 6.2.
\begin{lem}\label{min}
The functional $\iota|_{\Psi^\omega(\cO)}$ to is proper, and the minimum point of this restriction is unique.
\end{lem}
\begin{proof}
Let $X$ be a holomorphic vector field, such that $\displaystyle X\in \Psi^{\omega}_\star(T_{\cO, g})$; by a result due to 
T.~Mabuchi (cf. \cite{M}, page 238), if we define 
\begin{equation}\label{equa10}
g_t:= \exp(tX_{\bR})
\end{equation}
where $X_{\bR}$ is the real part of the vector $X$, then the map 
\begin{equation}\label{equa11}
t\to \Psi^\omega(g_t)
\end{equation}
is a smooth geodesic. Therefore, our statement is a consequence of the strict convexity properties of the functional $\iota$, combined with Proposition \ref{tangent}. 
\end{proof}
\medskip

\noindent We recall a few notations and results taken from \cite{Ca2}. Let $f$ be a smooth function on $M$;
we define
\begin{equation}\label{equa12}
Lf:= \dbar \nabla^{1,0}_\varphi f
\end{equation}
which written in coordinates gives
\begin{equation}\label{equa13}
Lf= \frac{\partial}{\partial z^{\ol \beta}}\Big(g^{\alpha\ol\mu}\frac{\partial f}{\partial z^{\ol \mu}}\Big)
\frac{\partial}{\partial z^{\alpha}}\otimes dz^{\ol \beta},
\end{equation}
and let $L^\star$ be the adjoint operator. Let $\cD_\varphi:=L^\star L $ be the Lichnerowicz operator. 
Then $\cD_\varphi$ is a self-adjoint
elliptic operator on the space of smooth complex functions of $M$, which can be written as
\begin{equation}\label{equa14}
\cD_\varphi f= \Delta_\varphi^2 f+ \langle \sqrt{-1}\ddbar f, \Ric_{\varphi}\rangle_{\omega_\varphi} + 
\langle \partial R_{\varphi} , \dbar f \rangle_{\omega_\varphi}
\end{equation}
where $\Ric_{\varphi}$ denotes the Ricci curvature of the metric $\omega_\varphi$, and $R_\varphi$ is its trace, namely the scalar curvature. Also, $\Delta_\varphi$ is the Laplace operator corresponding to $\omega_{\varphi}$.
\smallskip

\noindent Notice that if the metric $\omega_{\varphi}$ is cscK, the $\cD_{\varphi}$ is real, self-adjoint
operator. And we have the following result, consequence of the general elliptic theory.

\begin{lem}\label{elliptic} 
The operator $\cD_\varphi: \cC^\infty(M)\to \cC^\infty(M)$ has the following properties.
\begin{enumerate}

\item[{\rm (1)}] Its kernel coincides with the subspace of functions $f$ such that $\nabla^{1,0}_\varphi f$ is holomorphic.
\item[{\rm (2)}] The image $\cD_\varphi\big(\cC^\infty(M)\big)\subset \cC^\infty(M)$ is closed, and we have the orthogonal decomposition
\begin{equation}\label{equa15}
\cC^\infty(M)= \cD_\varphi\big(\cC^\infty(M)\big)\oplus \Ker(\cD_\varphi).
\end{equation}
\end{enumerate}
\end{lem} 
The point (1) is due to the compactness of $M$; the fact that the image of $\cD_\varphi$ is closed follows from Sobolev and G\"arding results for which we refer to
L.~H\"ormander \cite{Hor}.

%First of all, notice $\langle \sqrt{-1} \partial \bar \partial \langle \partial v, \bar \partial \xi\rangle_{\varphi} , Ric_{\varphi}\rangle_{\varphi} = (v_{,\delta} \xi_{,\bar \delta})_{, \bar \alpha \beta} (Ric_{\varphi})_{\alpha \bar \beta}  = (v_{,\delta} \xi_{,\bar \delta})_{, \beta \bar \alpha} (Ric_{\varphi})_{\alpha \bar \beta}$ since $(v_{,\delta }\xi_{,\bar\delta})$ is a function. We use the later expression(see (20)(21)(22)(23)) so that we won't get extra curvature terms when interchanging orders of differentiation. We may also use the other expression but it just leads to more terms 
%$$(v_{,\delta} \xi_{,\bar \delta})_{, \bar \alpha \beta}  = v_{, \delta \bar \alpha} \xi_{,\bar \delta \beta} + v_{, \delta \bar \alpha \beta} \xi_{, \bar \delta} +  v_{, \delta} \xi_{, \bar \delta  \bar\alpha \beta} $$ 
%and after correct interchange of orders of differentiation for the last two terms in the above equation, we can get the same results as in (22) and the extra terms(18) will cancel out.

\medskip

\noindent For any function $\varphi \in \mathcal{H}^{\infty}(M)$, we define a bilinear operator
$\displaystyle B_\varphi(\cdot, \cdot)$ acting on $u, v\in C^{\infty}(M)$ as follows
\begin{align*}\label{eqn10}
B_{\varphi}(u, v) &:= \langle  \partial \bar{\partial} v , \partial \bar{\partial}\Delta_{\varphi} u \rangle_{\varphi} + \Delta_{\varphi}\langle  \partial \bar{\partial} v , \partial \bar{\partial}u\rangle_{\varphi} +\langle  \partial \bar{\partial} \Delta_{\varphi} v , \partial \bar{\partial}u\rangle_{\varphi}\\
&+u_{,\bar{\alpha}p} v_{,\beta \bar{p}} (\Ric_{\varphi})_{\alpha\bar{\beta}} 
+u_{,\bar{p}\beta} v_{,p\bar{\alpha}} (\Ric_{\varphi})_{\alpha\bar{\beta}}.
\end{align*}

\noindent We end this introductory paragraph with the following statement, which will play an important role in 
our proof. In ordre to simplify the writing, we will use the following notation
\begin{equation}\label{equa40}
\langle \partial v_1, \ol\partial v_2\rangle_\varphi= \sum_{\alpha, \beta}\frac{\partial v_1}{\partial z_\alpha} 
\frac{\partial v_2}{\partial z_{\ol\beta}}g^{\alpha\ol\beta} 
\end{equation}
where $v_1, v_2$ are smooth functions on $M$. We have the next technical statement.

\begin{lem}\label{lem3} Let $\omega_\varphi\in [\omega]$ be an extremal metric, and
let $v, \xi$ be real-valued two smooth functions, such that $\cD_\varphi v = \bar \cD_{\varphi} v = 0$; then we have the next identity.
\begin{equation}\label{equa41}
\cD_\varphi\langle \partial v, \bar \partial \xi\rangle_\varphi = \langle \partial v, \dbar \cD_\varphi \xi\rangle_\varphi+ 
B_\varphi(v, \xi).
\end{equation}
\end{lem}

\begin{proof}
We check next the validity of \eqref{equa41} by a brute-force computation; we first assume that the scalar curvature of metric $\omega_\varphi$ is constant. Then we have the following long sequence of relations together with some explanations when passing from one line to another.
\begin{align}
\cD_\varphi\langle \partial v, \bar \partial \xi\rangle_\varphi  &= \Delta_{\varphi}^2\langle \partial v, \bar\partial \xi\rangle_{\varphi} + \langle \sqrt{-1} \partial \bar \partial \langle \partial v, \bar \partial \xi\rangle_{\varphi}, Ric_{\varphi}\rangle_{\varphi}\\
& = (v_{,\delta} \xi_{, \bar \delta})_{, \alpha \bar \alpha \beta \bar \beta} + (v_{,\delta} \xi_{,\bar \delta})_{, \beta \bar \alpha} (Ric_{\varphi})_{\alpha \bar \beta} \\
& = (v_{,\delta} \xi_{,\bar \delta \alpha})_{, \bar\alpha \beta \bar \beta} + (v_{,\delta} \xi_{, \bar \delta \beta})_{, \bar \alpha} (Ric_{\varphi})_{\alpha \bar \beta}\\
& = (v_{, \delta \bar \alpha} \xi_{, \bar \delta \alpha} + v_{, \delta } \xi_{, \bar \delta \alpha \bar \alpha})_{,\beta \bar \beta} + (v_{, \delta \bar \alpha} \xi_{,\bar \delta \beta} + v_{, \delta} \xi_{, \bar \delta \beta \bar\alpha} )( Ric_{\varphi})_{ \alpha \bar \beta}\\
& (\text{since } \xi_{,\bar \delta \alpha \bar \alpha} = \xi_{, \alpha \bar \delta \bar \alpha} = \xi_{, \alpha \bar \alpha \bar \delta}, \xi_{, \bar \delta \beta \bar \alpha} = \xi_{, \beta \bar \delta\bar \alpha} = \xi_{, \beta \bar \alpha \bar \delta})\\
& = (v_{, \delta \bar \alpha} \xi_{, \bar \delta \alpha} + v_{, \delta } \xi_{, \alpha \bar \alpha \bar \delta})_{,\beta \bar \beta} + v_{, \delta \bar \alpha} \xi_{,\bar \delta \beta} ( Ric_{\varphi})_{\alpha \bar \beta} + v_{, \delta} \xi_{, \beta \bar \alpha \bar \delta} (Ric_{\varphi})_{\alpha \bar \beta} \\
& = \Delta_{\varphi} \langle \partial \bar \partial v , \partial \bar \partial \xi\rangle_{\varphi} + \langle \partial \bar \partial v, \partial \bar \partial \Delta_{\varphi } \xi \rangle_{\varphi}  + v_{, \delta} \xi_{, \alpha\bar \alpha \bar \delta \beta \bar \beta}+  v_{, \delta \bar \alpha} \xi_{,\bar \delta \beta} ( Ric_{\varphi})_{\alpha \bar \beta}\\
 &+ v_{, \delta}\big( \xi_{,\beta \bar \alpha} (Ric_{\varphi})_{\alpha \bar \beta} \big)_{,\bar \delta}  - v_{, \delta}\xi_{,\beta \bar \alpha} (Ric_{\varphi})_{\alpha \bar \beta ,\bar \delta}\\
&(\text{since } \xi_{, \alpha\bar \alpha \bar \delta \beta \bar \beta} = (\Delta_{\varphi }\xi)_{, \bar \delta \beta \bar \beta} =  (\Delta_{\varphi }\xi)_{,  \beta  \bar \delta \bar \beta}  = (\Delta_{\varphi }\xi)_{,  \beta   \bar \beta\bar \delta} = (\Delta_{\varphi}^2 \xi)_{, \bar \delta})\\
& = \Delta_{\varphi} \langle \partial \bar \partial v , \partial \bar \partial \xi\rangle_{\varphi} + \langle \partial \bar \partial v, \partial \bar \partial \Delta_{\varphi } \xi \rangle_{\varphi}  + v_{, \delta} (\Delta_{\varphi}^2 \xi + \xi_{,\beta \bar \alpha} (Ric_{\varphi})_{\alpha \bar \beta})_{, \bar \delta}\\
&+  v_{, \delta \bar \alpha} \xi_{,\bar \delta \beta} ( Ric_{\varphi})_{\alpha \bar \beta} 
- v_{, \delta}\xi_{,\beta \bar \alpha} (Ric_{\varphi})_{\alpha \bar \beta ,\bar \delta}
\end{align}
As a consequence, we infer that we have
\begin{align}
&\cD_\varphi\langle \partial v, \bar \partial \xi\rangle_\varphi- \langle \partial v, \bar \partial \cD_\varphi \xi\rangle_\varphi\\
& = \Delta_{\varphi} \langle \partial \bar \partial v , \partial \bar \partial \xi\rangle_{\varphi} + \langle \partial \bar \partial v, \partial \bar \partial \Delta_{\varphi } \xi \rangle_{\varphi}  +  v_{, \delta \bar \alpha} \xi_{,\bar \delta \beta} ( Ric_{\varphi})_{\alpha \bar \beta} - v_{, \delta}\xi_{,\beta \bar \alpha} (Ric_{\varphi})_{\alpha \bar \beta ,\bar \delta}\\
& = B_{\varphi}(v, \xi).  
\end{align}
In order to establish this equality, we have used the identity 
\begin{align*}
\langle \partial \bar{\partial} \Delta_{\varphi} v, \partial \bar{\partial} \xi  \rangle_{\varphi} &= v_{, \alpha \bar{\alpha} \beta \bar{\delta}} \xi_{, \bar{\beta} \delta}\\
&= (v_{, \alpha \beta \bar{\alpha}} - R_{\beta \bar{\alpha} \alpha \bar{l}} v_{, l})_{, \bar{\delta}} \xi_{,\bar{\beta} \delta}\\
& = - R_{\beta \bar{l}} v_{, l\bar{\delta}} \xi_{, \bar{\beta} \delta} - R_{\beta \bar{l}, \bar{\delta}} v_{, l} \xi_{, \bar{\beta \delta}}.
\end{align*}
This completes the proof of the lemma, in the case of a cscK metric (we remark that we are only
using this hypothesis in the expression of the operator $\cD_\varphi$ in the first line of the long string of equalities above).
\smallskip

In the preceding computations, if $\omega_{\varphi}$ is any K\"ahler metric (i.e. no curvature assumptions), 
then the expression of $\cD_\varphi$ has an additional a term, containing the derivative of scalar curvature. In order to complete the proof, we still have to check that we have
$$(R_{\varphi})_{, \delta}( \langle \partial v , \bar \partial \xi \rangle_{\varphi})_{,\bar \delta} - \langle \partial v , \bar \partial \big( (R_{\varphi})_{,\delta} \xi_{, \bar \delta}\big) \rangle_{\varphi} =0.$$
Here we will use the curvature assumption, namely that $\omega_{\varphi}$ is extremal, because then we have
\begin{align}
&(R_{\varphi})_{,\delta} (v_{,\alpha} \xi_{,\bar \alpha} )_{,\bar \delta} - v_{, \alpha} \big( (R_{\varphi})_{,  \delta} \xi_{,\bar \delta}\big)_{,\bar \alpha} \\
& = (R_{\varphi})_{,\delta} v_{,\alpha} \xi_{,\bar \alpha \bar \delta} + (R_{\varphi})_{,\delta} v_{,\alpha \bar\delta} \xi_{,\bar \alpha }- v_{, \alpha} (R_{\varphi})_{, \delta} \xi_{, \bar \delta \bar \alpha} - v_{,\alpha} (R_{\varphi})_{,  \delta \bar \alpha} \xi_{, \bar \delta}\\
&= \big((R_{\varphi})_{,\delta} v_{,\alpha \bar\delta} - v_{,\delta} (R_{\varphi})_{,  \alpha \bar \delta}\big)\xi_{,\bar \alpha }\\
& = \big((R_{\varphi,\delta} v_{, \bar \delta})_{, \alpha} - R_{\varphi, \delta \alpha} v_{,\bar \delta} - (v_{, \delta} R_{\varphi, \bar \delta})_{, \alpha}\big) \xi_{,\bar \alpha} \\
& = -  R_{\varphi, \delta \alpha} v_{,\bar \delta} \xi_{,\bar \alpha} + (\cD_{\varphi} v - \bar \cD_{\varphi} v )_{,\alpha}\xi_{,\bar \alpha} = 0
\end{align}
The proof of Lemma \ref{lem3} is therefore finished.
\end{proof}
\bigskip
%%%%%%%%%%%%%%%%%%%%%%%%%%%%%%%%%%%%%%%%%%%%%%%%%%%%%%%%%%%%%%%%%%%%%%%%%%%%%%%%%%%%%%%%%%%%%%%%%%%%%%%%%%%%%%%%%%%%%%%%%%%%%%%%%%%%%%%%%%%%%%%%%%%%%%%%%%%%%%%%%%%%%%%%%%%%%%%%%

\section{Proof of Theorem $\ref{thm1.1}$}\label{sect3}
We are now ready to prove Theorem $\ref{thm1.1}$, concerning the deformations 
of K\"ahler metrics with constant scalar curvature. This will be achieved by the implicit function theorem; to start with, we define
the functional space 
\begin{align*}
\mathcal{H}^{4,\alpha}(M) = \{\varphi \in C^{4,\alpha}(M, \mathbb{R}) | \omega_{\varphi} =  \omega + \sqrt{-1} \partial \bar{\partial} \varphi > 0\}.
\end{align*}

\noindent The continuity path we will use is the same as the one in \cite{CP}, namely
\begin{align*}
\cF: \mathcal{H}^{4,\alpha} (M) \times [0,1] &\longrightarrow C^{\alpha}(M) \times [0,1]\\
\cF(\varphi, t) & = \big(R_{\varphi} - \underline{R} -(1-t)(\text{tr}_{\varphi} \omega - n), t\big)
\end{align*}
where $R_{\varphi}$ is the scalar curvature of $\omega_{\varphi}$ and
$$\underline{R} = \frac{1}{\Vol(X, \omega)}\int_XR_\varphi\omega_\varphi^n$$
is the average of the scalar curvature (which is easily seen to be a cohomological quantity). The first component of $\cF$ will be denoted in what follows by $F$, i.e.
 
\begin{equation}
F(\varphi, t):= R_{\varphi} - \underline{R} -(1-t)(\text{tr}_{\varphi} \omega - n).
\end{equation}
\medskip
In this section, our main result states as follows.
\begin{thm} \label{main} Let $(M, \omega)$ be a compact K\"ahler manifold, such that 
the scalar curvature of $\omega$ is constant. We denote by $\varphi_1\in \wt \cH$ the potential for which the
restriction $\displaystyle \iota |_{\Psi^\omega(\cO)}$ 
is minimal; let $\omega_{\varphi_1}\in [\omega]$ be the corresponding metric. Then there exists $\epsilon > 0 $, such that for any $1-\epsilon< t \leq 1$, there exists $\varphi_t=\varphi(t, \cdot)$ satisfying $$F(\varphi_t,t) = 0,$$ 
and such that $\phi(1, \cdot)$ coincides with the potential $\varphi_1$.
\end{thm}

\begin{proof}
As we have already mentioned, we intend to use the implicit function theorem, so the first thing to do 
would be to compute the differential of $\cF$ at the point $(\varphi_1, 1)$ for which we have
$\cF(\varphi_1,1) = 0$. A standard calculation (which will not be detailed here) shows that we have
\begin{align*}
d\cF_{(\varphi_1, 1)}: C^{4,\alpha}(M) \times \mathbb{R}& \longrightarrow C^{\alpha}(M)  \times \mathbb{R}\\
(u, s) &\longmapsto    (-\cD_{\varphi_1}u + s(\text{tr}_{\varphi_1} \omega - n), s),
\end{align*}
where $\cD_{\varphi_1}$ is the Lichnerowicz operator  
with respect to $\omega_{\varphi_1}$ defined in the previous section (we are using here the fact that the  
scalar curvature of $\omega_{\varphi_1}$ is constant).

Let $u_0$ be a smooth function such that $\cD_{\varphi_1}(u_0)= 0$ (we notice that in our set-up, the kernel of 
$\cD_{\varphi_1}$ has strictly positive dimension); then we have $\displaystyle d\cF_{(\varphi_1, 1)}(u_0, 0)= 0$. Also, we remark that thanks to the minimality property of $\varphi_1$, we have 
\begin{equation}\label{eq100}
\int u (\text{tr}_{\varphi_1} \omega - n) \omega_{\varphi_1}^n = 0 
\end{equation}
for any $\displaystyle u \in \Ker(\cD_{\varphi_1})$: this is a consequence of Proposition \ref{tangent}, combined with the
relation \eqref{grad}. In conclusion, $\displaystyle d\cF_{(\varphi_1, 1)}$ is neither injective nor surjective.
\smallskip

\noindent Let $k$ be a positive integer; we introduce the following notations. 
\begin{align*}
\mathcal{H}_{\varphi_1} &= \{ u \in C^{\infty}(M) | \cD_{\varphi_1}(u) = 0 ,\int u \omega_{\varphi_1}^n = 0  \} \\
\mathcal{H}_{\varphi_1, k}^{\perp} &= \{u\in C^{k, \alpha}(M) | \int u \omega_{\varphi_1}^n =0, \int u v \omega_{\varphi_1}^n =0 , \text{for all }    v \in \mathcal{H}_{\varphi_1} \}.
\end{align*}
Thus we have the decomposition $C^{k,\alpha} (M) =\mathbb{R} \oplus \mathcal{H}_{\varphi_1} \oplus 
\mathcal{H}_{\varphi_1, k}^{\perp} $. 
By using these notations the relation \eqref{eq100} becomes
 \begin{equation}\label{eq101}
 \text{tr}_{\varphi_1} \omega - n \in \mathcal{H}_{\varphi_1, 0}^{\perp}.
 \end{equation}
 
 \noindent Consider the following projection map 
\begin{align*}
\Pi: ( \mathbb{R} \oplus \mathcal{H}_{\varphi_1} \oplus \mathcal{H}_{\varphi_1,4}^{\perp}) \times [0,1] &\longrightarrow (\mathbb{R} \oplus \mathcal{H}_{\varphi_1} \oplus \mathcal{H}_{\varphi_1,0}^{\perp}) \times [0,1] \\
(a+u + w ,t)  & \longmapsto ( a+ u + \pi_2\circ F( \varphi_1 + a+ u + w, t) ,t ),
\end{align*}
where $\pi_2$ is the projection from $C^{\alpha}(M)$ to $\mathcal{H}_{\varphi_1,0}^{\perp}$. The 
derivative of $\Pi$ at $(0, 1)$ equals
\begin{align*}
d\Pi_{(0,1)}: (\mathbb{R} \oplus \mathcal{H}_{\varphi_1} \oplus \mathcal{H}_{\varphi_1,4}^{\perp}) \times \mathbb{R} &\longrightarrow (\mathbb{R} \oplus \mathcal{H}_{\varphi_1} \oplus \mathcal{H}_{\varphi_1,0}^{\perp}) \times \mathbb{R}\\
( a+ u + w ,s)  & \longmapsto ( a+ u -\cD_{\varphi_1} w + s (\text{tr}_{\varphi_1} \omega- n) ,s). 
\end{align*}
\smallskip

The relation $\text{tr}_{\varphi_1} \omega- n \in \mathcal{H}_{\varphi_1, 0}^{\perp}$ 
combined with Lemma $\ref{elliptic}$ show that $d\Pi|_{(\varphi_1,1)}$ is bijective. By the 
inverse function theorem,
%$$\Pi^{-1}: V_{(0,1)} \subset (\mathbb{R} \oplus \mathcal{H}_{\varphi_1} \oplus \mathcal{H}_{\varphi_1,0}^{\perp}) \times [0,1]  \rightarrow U_{(\varphi_1,1)} \subset(\mathbb{R} \oplus \mathcal{H}_{\varphi_1} \oplus \mathcal{H}_{\varphi_1,4}^{\perp}) \times [0,1] .$$ 
given any $\|u\|_{C^{\alpha}(M)} < \epsilon$ and $|t-1| < \epsilon$ we obtain $\psi(u, t)$ such that 
%$\pi'_2$ is the projection from $C^{4,\alpha}(M)$ to $\mathcal{H}_{\varphi_1,4}^{\perp}$. Therefore 
\begin{equation}\label{eq102}
\pi_2 \circ F(\varphi_1 +u + \psi(u,t), t) = 0.
\end{equation}
The equality \eqref{eq102} shows that we have 
\begin{equation}
-\cD_{\varphi_1} \frac{\partial \psi}{\partial t}\big\vert_{(0,1)} + \text{tr}_{\varphi_1} \omega - n = 0 
\end{equation}
by differentiating with respect to $t$. Also, the derivative of \eqref{eq102} with respect to $u$ gives
\begin{equation}\label{42}
\frac{\partial \psi}{\partial u}\big\vert_{(0,1)}( v) = 0, 
\end{equation}
for any $v \in \mathcal{H}_{\varphi_1}$.

%The equality \eqref{eq102} can be rewritten as
%$\pi_1\circ F(\varphi_1 + u + \psi(u,t), t)= F(\varphi_1 + u + \psi(u,t), t)$ . Thus we have

\noindent We introduce the functional
\begin{equation}
P(u,t):= \pi_1\circ F(\varphi_1 + u + \psi(u,t), t)
\end{equation}
where $\pi_1$ is the projection onto the factor $\mathcal{H}_{\varphi_1}$.
In order to finish the proof, it remains to solve the equation 
$$P(u_t, t) =0$$ 
for each $1-\varepsilon< t\leq 1$.
However, we cannot apply the implicit function theorem, because
it turns out that $P(u,1) = 0$ for any $u\in \mathcal{H}_{\varphi_1}$. Indeed, the differential of $P$ with respect to $u$ vanishes at each point $(u, 1)$ (this is a consequence of \eqref{42}, combined with the fact that $P(0, 1)= 0$).

\noindent Then we consider the ``first derivative"
\begin{equation}\label{44}
\widetilde{P}(u,t): = \frac{P(u,t)}{t-1} 
\end{equation}
and we observe that $\widetilde{P}(u, t)$ can be extended as a continuous function on $\mathcal{H}_{\varphi_1} \times [0,1]$, because of the equality
\begin{align*}
\widetilde{P}(u,1) = \lim_{t \rightarrow 1^{-}} \frac{P(u,t)}{t-1}  = \frac{\partial P}{\partial t}\big|_{(u,1)}.
\end{align*}

\noindent Our next observation is that it would be enough to solve the equation $\wt P(u_t, t)= 0$, and so
we will compute the partial derivative $\displaystyle \frac{\partial \wt{P}}{\partial u}\big|_{(0,1)}$ and we will show that it is invertible. Prior to this, we re-write the expression of $\wt P$ as follows.
\begin{align*}
\wt{P}(u,1) = \frac{\partial}{\partial t} P|_{(u,1)} &= \pi_1[-\cD_{\varphi_1 + u + \psi_{u,1}} \frac{\partial \psi}{\partial t}\big|_{(u,1)} + \text{tr}_{\varphi_1 + u + \psi_{u,1}}\omega - n]\\
& =\pi_1 [- \Delta_{\varphi_1 + u + \psi_{u,1}}^2 \frac{\partial\psi}{\partial t}\big|_{(u,1)} - \big(\frac{\partial\psi}{\partial t}\big|_{(u,1)}\big)_{,\bar{\alpha} \beta} (Ric_{\varphi_1 + u + \psi_{u,1}})_{\alpha \bar{\beta}} \\
&\qquad\qquad\quad + \text{tr}_{\varphi_1 + u + \psi_{u,1}}\omega - n]
\end{align*}
\medskip

\noindent We compute
\begin{align*}
\frac{\partial}{\partial u} \wt{P} |_{(0,1)} (v)&= \pi_1\{\langle  \partial \bar{\partial} v , \partial \bar{\partial}\Delta_{\varphi_1} \xi \rangle_{\varphi_1} + \Delta_{\varphi_1}\langle  \partial \bar{\partial} v , \partial \bar{\partial}\xi\rangle_{\varphi_1} +\langle  \partial \bar{\partial} \Delta_{\varphi_1} v , \partial \bar{\partial}\xi\rangle_{\varphi_1}+\xi_{,\bar{\alpha}p} v_{,\bar{p}\beta} (Ric_{\varphi_1})_{\alpha\bar{\beta}} \\
&\qquad +\xi_{,\bar{p}\beta} v_{,p\bar{\alpha}} (Ric_{\varphi_1})_{\alpha\bar{\beta}} - \langle\partial\bar{\partial} v, \chi\rangle_{\varphi_1}  -\cD_{\varphi_1}\frac{\partial^2\psi}{\partial u \partial t}|_{(0,1)}(v)\}\\
&= \pi_1[B_{\varphi_1} (v, \xi) - \langle\partial\bar{\partial} v, \chi\rangle_{\varphi_1}]
\end{align*}
where $\xi = \frac{\partial\psi}{\partial t}|_{(0,1)} $ and $B_{\varphi_1}(v,\xi)$ is the operator in Lemma 
$\ref{lem3}$. The previous string of equalities combined with Lemma $\ref{lem3}$ imply that we have
\begin{align*}
\frac{\partial}{\partial u} \wt{P} |_{(0,1)} (v) &= \pi_1[ \cD_{\varphi_1}(\langle \partial v, \bar{\partial} \xi \rangle_{\varphi_1}) - \langle \partial v, \bar{\partial} \cD_{\varphi_1} \xi \rangle_{\varphi_1} - \langle\partial\bar{\partial} v, \omega\rangle_{\varphi_1}]\\
& = \pi_1( - \langle \partial v, \bar{\partial} (\text{tr}_{\varphi_1} \omega - n) \rangle_{\varphi_1} - \langle\partial\bar{\partial} v, \omega\rangle_{\varphi_1}).
\end{align*}

\noindent Then we see that the scalar product 
\begin{align*}
\int\frac{\partial \wt{P}}{\partial u}\big|_{(0,1)} (v) v \omega_{\varphi}^n &= \int (-\langle \partial v, \bar{\partial} ( \text{tr}_{\varphi_1} \omega - n)\rangle_{\varphi_1}v -  \langle\partial\bar{\partial }v, \omega\rangle_{\varphi_1} v) \omega_{\varphi_1}^n\\
&= \int v_{,\bar{\alpha}} v_{, \beta} \omega_{\alpha\bar{\beta}} \omega_{\varphi_1}^n \geq 0,
\end{align*}
is positive, and it is equal to zero if and only if $v=0$ in $\mathcal{H}_{\varphi_1}$.
Therefore, $\displaystyle \frac{\partial \wt{P}}{\partial u} \big|_{(0,1)}$ is injective and therefore bijective. 
The implicit function theorem shows that there exists $u_t$ such that $P(u_t, t) = 0$ for t sufficiently close to 1; when combined with \eqref{eq102}, this implies
\noindent 
\begin{align*}
F(\varphi_1 + u_t + \psi({u_t,t}) , t) = 0\end{align*}
which is what we wanted to prove.

\end{proof}
\medskip

\noindent The uniqueness of constant scalar curvature metrics follows almost immediately.
\begin{cor}
Suppose there exists two cscK metrics $\omega_{\varphi_1} , \omega_{\varphi_2}\in [\omega]$. Then there exists an element $\sigma \in Aut_0(M)$ such that $\sigma^* \omega_{\varphi_1} = \omega_{\varphi_2}$.
\end{cor}
\begin{proof}
We argue by contradiction: suppose we have two cscK orbits $\mathcal{O}_1$ and $\mathcal{O}_2$ such that $\mathcal{O}_1 \neq \mathcal{O}_2$. Then we consider the K\"ahler potentials $\varphi_1$ and $\varphi_2$ for which the restriction of $\iota$ to $\mathcal{O}_1$ and $\mathcal{O}_2$ is reached,
respectively,
 
By Theorem $\ref{main}$, we obtain two paths $\varphi_i (t)$ with $\varphi_k(1, \cdot) =\varphi_k$, for
$k = 1, 2$; moreover, we obtain  
\begin{align}\label{eqn12}
R_{\varphi_k(t)} -\underline{R} - (1-t)(\text{tr}_{\varphi_k(t)} \omega- n) = 0.
\end{align}
As explained in the introduction, for fixed $t<1$, the solution of equation $(\ref{eqn12})$ is unique. Thus, for any $1-\epsilon < t < 1$, $\varphi_1(t) = \varphi_2(t)$. In particular $\varphi_1 = \varphi_2$. Therefore, we are done.\end{proof}

%%%%%%%%%%%%%%%%%%%%%%%%%%%%%%%%%%%%%%%%%%%%%%%%%%%%%%%%%%%%%%%%%%%%%%%%%%%%%%%%%%%%%%%%%%%%%%%%%%%%%%%%%%%%%%%%%%%%%%%%%%%%%%%%%%%%%%%%%%%%%%%%%%%%%%%%%%%%%%%%%%

\section{Twisted extremal K\"ahler metrics}

\medskip

\noindent We start with a general discussion about the proof of Theorem \ref{thm1.2} which will follow; hopefully, this will clarify a few facts/choices which will appear shortly.
\smallskip

Let $\displaystyle \omega_{\varphi_1}\in [\omega]$ be an extremal metric.
In order to prove Theorem \ref{thm1.2}, our strategy will be to determine the 
path $\varphi_t:= \varphi(t, \cdot)$ by
solving the equation 

\begin{equation}\label{0501}
\nabla^{1, 0}_{\varphi_t}\big(R_{\varphi_t} -(1-t) \tr_{\varphi_t} \omega\big)= X_1
\end{equation}
where $X_1:= \nabla^{1, 0}_{\varphi_1}(R_{\varphi_1})$ is a holomorphic vector field. 
\smallskip

\noindent We show next that the order of differentiation in the expression \eqref{0501} can be reduced. Indeed we have
\begin{equation}\label{0502}
i_{X_1}\omega_{\varphi_t}:= \sqrt{-1} \dbar \rho_t(X_1)
\end{equation}
for a unique function $\rho_t(X_1): M\to \bC$ normalized such that 
\begin{equation}\label{0503}
\int_X\rho_t(X_1)\omega_{\varphi_t}^n= 0.
\end{equation}
(this can be seen by writing $\omega_{\varphi_t}= \omega_{\varphi_1}+ \sqrt{-1}\ddbar \phi_t$). By combining 
\eqref{0501} and \eqref{0502}, the equation we have to solve is equivalent to
\begin{equation}\label{0504}
R_{\varphi_t} -\underline R-(1-t)(\tr_{\varphi_t}\omega- n)= \rho_t(X_1).
\end{equation}
The equation \eqref{0504} above is very similar to the one we had to deal with in the previous section. 
We could then simply follow the same procedure as in the proof of Theorem \ref{thm1.1} (i.e. start with an extremal metric whose potential minimizes the functional $\iota$ and so on) in order to conclude,
even if 
the presence of the factor $\rho_t(X_1)$ complicates a bit the situation, as we will see next.
However in doing so, we would not be able to obtain the uniqueness statement Corollary \ref{cor1.4}, for a simple reason which will become obvious at the end of this section (basically we need the holomorphic gradient
of the scalar curvature corresponding to $\omega_1$ and $\omega_2$ to coincide). Also the term $\rho_{t} (X_1)$ would in general be complex valued and we don't want to choose our image space to be complex valued functions.

%The solution of the equation \eqref{0504} is a critical point of the \emph{modified K-energy functional}, 
%\begin{equation}\label{eqn4.11}
%\cE_{X_1}+ (1-t)\iota
%\end{equation}
%where $\displaystyle \cE_{X_1}$ is 
%defined
%via its derivative as follows
%\begin{align*}
%d\cE_{X_1}|_\varphi(\dot \varphi) &= \int_M \big(- (R_{\varphi} - \underline{R}) + \rho_{\varphi} (X_1)\big)\dot\varphi\omega_{\varphi}^n.
%\end{align*}
%The functional in \eqref{eqn4.11} is strictly convex along weak geodesics, but it depends on $X_1$, so a-priori we cannot argue as before.
\medskip

\noindent Luckily, it is possible to bypass these difficulties by using the following results; the first is 
due to E. Calabi.

\begin{thm} \label{aut_cal}
{\rm \cite{Ca2}}
For any extremal K\"ahler metric $g$ in a compact complex manifold $M$, the identity component 
$\Iso_0(M, g)$ of the group of holomorphic isometries of $(M, g)$ coincides with a maximal compact connected subgroup of $\Aut_0(X)$.
\end{thm}

\noindent The following statement is a reformulation of a result due to Futaki-Mabuchi, cf. \cite{FM},
in which we are using Theorem \ref{aut_cal}.
\begin{thm} \label{uniq_vf}
{\rm \cite{FM}}
Let $g_j\in [\omega]$ be two extremal metrics, such that 
$$\Iso_0(M, g_1)= \Iso(M, g_2).$$
Then we have $\displaystyle \nabla^{1,0}_{g_1}(R_{g_1})= \nabla^{1,0}_{g_2}(R_{g_2})$.
\end{thm}
\medskip

\noindent We assume next that $\omega$ is an extremal metric, and 
we denote by $K:= \Iso(M, \omega)$
the corresponding group of holomorphic isometries. The next step would be to consider the 
minimum $\displaystyle \omega_{\varphi_1}$ of the restriction of the functional $\iota$ to the 
space of potentials $\Psi^\omega(\cO)$ corresponding to $\omega$; in doing so, it is possible that the isometry group of $\displaystyle \omega_{\varphi_1}$ is different from $K$. 

In order to prevent this to happen, we will restrict the 
functional $\iota$ to the space of $\omega$-potentials which are $K$-invariant,
 defined as follows
\begin{align*}
\mathcal{H}_K^{\infty} (M) &= \{\varphi \in \mathcal{H}^{\infty}(M)| \varphi = \varphi \circ \sigma \text{ for any } \sigma \in K\};\\
\mathcal{H}_K^{k,\alpha} (M) &= \{\varphi \in \mathcal{H}^{k, \alpha}(M)| \varphi = \varphi \circ \sigma \text{ for any } \sigma \in K\};
\end{align*}
we equally consider the space
\begin{align*}
C_K^{k,\alpha}(M) &= \{ u \in C^{k,\alpha} (M)|u = u \circ \sigma \text{ for any } \sigma \in K\ \}.
\end{align*}
Let $\cO_K$ be the quotient $N_K/K$, where we denote by 
$N_K$ the normalizer of $K$ in $\Aut_0(M)$, that is to say the group consisting of $g\in \Aut_0(M)$ such that $gKg^{-1}= K$. 
\medskip

\noindent We have the following statement, which is the analogue of Proposition
\ref{tangent}.
\begin{prop}\label{Ktangent}
Let $(M, \omega)$ be a compact K\"ahler manifold, such that $\omega$ is extremal. Then the image of the tangent space $\displaystyle (\Psi^\omega)_\star(T_{\cO_K, g})$ coincides with 
the space generated by the real-valued functions $f\in \cC^\infty(M)$ which are $K$-invariant, such that $\nabla^{1,0}_\varphi f$ is holomorphic,
where $\varphi:= \Psi^\omega(g)$.
\end{prop}
\begin{proof}
First, we have to check that the image of $\displaystyle \Psi^\omega|_{\cO_K}$ consists of $K$-invariant potentials. Let $g\in N_K$; we have
\begin{equation}\label{0506}
g^\star\omega= \omega+ \sqrt{-1}\ddbar \varphi.
\end{equation}
and let $\sigma\in K$. Since $g$ belongs to the normalizer of $K$, we have
$$\sigma^\star g^\star\omega= g^\star\omega$$
hence by \eqref{0506} we obtain $\varphi\circ\sigma= \varphi$.

Let $g_t$ be a smooth path in $N_K$, such that 
$g_0= g$ and such that the derivative $\displaystyle \frac{dg_t}{dt}\big|_{t=0}$ identifies with a holomorphic vector field which we denote by $X$. There exists a smooth family $(\varphi_t)\subset \wt \cH$ with $\varphi_0 = \varphi$ so that we have 
\begin{equation}\label{equa50}
g_t^\star\omega:= \omega_{\varphi_t}
\end{equation}
for any parameter $t$. Since $\omega$ is extremal and $K$-invariant, so is $\omega_{\varphi_0}$. By Calabi's theorem\cite{Ca2}, we can decompose a holomorphic (1,0) vector field as follows
\begin{align*}
T_e Aut_0(M) = a(M) \oplus \nabla^{1,0}_{\varphi_0} E,
\end{align*}
where $a(M)$ are autoparallel vectors on $(M, \omega_{\varphi_0})$ and E is the kernel of Lichnerowicz derivative, i.e. $E=\{f \in C^{\infty}(M, \mathbb{C}) | \cD_{\varphi_0}f = 0\}$. 

We denote by $\bar{\cD}_{\varphi_0}$  the conjugate of the operator $\cD_{\varphi_0}$. 
Since the metric $\displaystyle \omega_{\varphi_0}$ is extremal, we have
$$[\cD_{\varphi_0}, \bar \cD_{\varphi_0}]= 0$$
i.e. the two operators commute. In particular, we can further decompose the space
$E$ according to the eigenspaces of $\bar\cD_{\varphi_0}|_{E}$, so that we have
\begin{align*}
T_e Aut_0(M) = a(M) \oplus \nabla^{1,0}_{\varphi_0} E_0 \oplus \sum_{\lambda > 0} \nabla^{1,0}_{\varphi_0} E_{\lambda},
\end{align*}
where $E_{\lambda}$ represents the $\lambda$-eigenspace of $\bar\cD_{\varphi_0}$. Notice here that the eigenvalues above are real and nonnegative.

By the above discussion, we can write 
\begin{align*}
X = X_a + \nabla_{\varphi_0}^{1,0}( f_0 + \sum_{\lambda> 0}  f_{\lambda})
\end{align*}
where $X_a \in a(M)$ and $f_{\lambda} \in E_{\lambda}$ for $\lambda \geq 0$. Notice here the sum is finite since $\Aut_0(M)$ is finite dimensional. Since $g_t \in N_K$, it implies that for any $\sigma \in K$,
\begin{align*}
g_t^* \sigma^* (g_t^{-1})^* \omega_{\varphi_0} = \omega_{\varphi_0}.
\end{align*}
Differentiate with respect to t, 
\begin{align*}
0= (\frac{\mathrm{d}}{\mathrm{d} t} g_t^* \sigma^* (g_t^{-1})^* \omega_{\varphi_0})|_{t=0} &= \sigma^* (\frac{\mathrm{d}}{\mathrm{d} t} (g_t^{-1})^* \omega_{\varphi_0})|_{t=0} + (\frac{\mathrm{d}}{\mathrm{d} t} g_t^* \sigma^*\omega_{\varphi_0})|_{t=0}\\
& = \sqrt{-1} \partial \bar{\partial}[( f+\bar{f}) - ( f+\bar{f})\circ \sigma]
\end{align*}
where $f = f_0 + \sum_{\lambda>0} f_{\lambda}$. Hence we get for any $\sigma \in K$
\begin{align}\label{eqn4.1}
f+ \bar{f}- (f + \bar{f}) \circ \sigma = 0.
\end{align}
Applying $\bar\cD_{\varphi_0}$ on both hand sides of $(\ref{eqn4.1})$ k times, we get that
\begin{align*}
\sum_{\lambda>0} \lambda^k  (f_{\lambda} - f_{\lambda} \circ \sigma) = 0.
\end{align*}
Thus we infer that $f_{\lambda} - f_{\lambda} \circ \sigma=0$ for any $\sigma \in K$. Consider
$$\Xi_{\varphi_0}: = \Im( \nabla_{\varphi_0}^{1,0} R_{\varphi_0}) = \frac{\sqrt{-1}}{2} [g_{\varphi_0}^{\alpha\bar{\beta}} R_{\varphi_0, \bar{\beta}}\frac{\partial}{\partial z_{\alpha}} - g_{\varphi_0}^{\alpha \bar{\beta}} R_{\varphi_0, \alpha} \frac{\partial}{\partial \bar{z_{\beta}}} ],$$
and $\exp(t\Xi_{\varphi_0})$ is a one parameter subgroup of $K$. Since $f_{\lambda}$ is $K$-invariant,
\begin{align*}
0 = \frac{d}{dt} \exp(t \Xi_{\varphi_0})^* f_{\lambda} = \Xi_{\varphi_0}( f_{\lambda}) = \frac{\sqrt{-1}}{2}[R_{\varphi_0, \bar{\delta}} f_{\lambda, \delta} - R_{\varphi_0, \delta} f_{\lambda, \bar{\delta}}].
\end{align*}
Hence 
\begin{align*}
\lambda f_{\lambda} = \bar\cD_{\varphi_0} f_{\lambda} = -(\cD_{\varphi_0} - \bar{\cD}_{\varphi_0} ) f_{\lambda}= R_{\varphi_0, \bar{\delta}} f_{\lambda, \delta} - R_{\varphi_0, \delta} f_{\lambda, \bar{\delta}} = 0.
\end{align*}
Therefore, $f_{\lambda } = 0$ for any $\lambda > 0$. Thus, 
$$X = X_a + \nabla^{1,0}_{\varphi_0} f_0$$
where $f_0 \in KerD_{\varphi_0} \cap Ker \bar D_{\varphi_0}$ is a $K$-invariant complex-valued function. Therefore, $\Re(f_0)$ and $\Im(f_0)$ are both $K$-invariant and belong to 
$\Ker \cD_{\varphi_0} \cap \Ker \bar\cD_{\varphi_0}$. By differentiating ($\ref{equa50}$) at $t=0$, we get that
$$\sqrt{-1} \partial \bar \partial \dot \varphi_0 = \sqrt{-1} \partial \bar \partial \Re(f_0). $$
The rest of the argument follows from Proposition $\ref{tangent}$ and it ends the proof.
\end{proof}
\medskip

\noindent Precisely as in Lemma \ref{min}, the restriction $\displaystyle \iota|_{\Psi^\omega(\cO_K)}$
is proper.
Let $\varphi_1\in \Psi^\omega(\cO_K)$ be the potential for which its minimum is reached; we denote by $\omega_{\varphi_1}$ the resulting (extremal) metric. We note that we have the equality
$$\Iso(M, \omega_{\varphi_1})= K$$
by Theorem \ref{aut_cal}.

Let $X_1 = \nabla_{\varphi_1}^{1,0} R_{\varphi_1} $ be the holomorphic vector field
corresponding to the metric $\omega_{\varphi_1}$. We define the functional
$\displaystyle \cF_K: \mathcal{H}_K^{4,\alpha} (M) \times [0,1] \longrightarrow C_K^{0,\alpha}(M)$ by the formula
\begin{equation}\label{0507}
\cF_K(\varphi, t)= R_{\varphi} - \underline{R} - (1-t) (\tr_{\varphi} \omega - n) - \rho_{\varphi}(X_1)
\end{equation}
where we recall that $\rho_{\varphi} (X_1)$ is uniquely determined by 
\begin{equation}\label{0508}
i_{X_1} \omega_{\varphi} =  \sqrt{-1} \bar{\partial} \rho_{\varphi} (X_1), \quad
\int_M \rho_{\varphi}(X_1) \omega_{\varphi}^n = 0.
\end{equation}
\medskip

\begin{remark}{\rm
If $\varphi \in \cH_K^{4,\alpha}(M)$, then $\rho_{\varphi}(X_1)$ is real-valued. This is because
\begin{align}
\sqrt{-1} \dbar \rho_{\varphi} (X_1) = i_{X_1}\omega_{\varphi} &= i_{X_1} \omega_{\varphi_1} + i_{X_1} \big(\sqrt{-1} \partial \dbar (\varphi - \varphi_1) \big)\\
&=\sqrt{-1} \dbar(R_{\varphi_1} + X_1(\varphi - \varphi_1)).
\end{align}
Thus
\begin{align}
\rho_{\varphi} (X_1) =R_{\varphi_1} + X_1(\varphi - \varphi_1) - \int_M \big(R_{\varphi_1} + X_1(\varphi - \varphi_1) \big) \omega_{\varphi}^n.
\end{align}
And the imaginary part of $\rho_{\varphi} (X_1)$ is given by
\begin{align}
\Im(\rho_{\varphi}(X_1)) = \Im(X_1) (\varphi -\varphi_1) - \int_M \Im(X_1) (\varphi -\varphi_1) \omega_{\varphi}^n.
\end{align}
On the other hand, we know that $\Im(X_1)$ is in the Lie algebra of $\Iso(M, \omega_{\varphi_1}) = K$. Since $(\varphi - \varphi_1)$ is $K$-invariant, we obtain $\Im(\rho_{\varphi} (X_1)) = 0$.
}
\end{remark}

\noindent By the discussion at the beginning of this section, the following perturbation theorem implies Theorem $\ref{thm1.2}$.

\begin{thm}\label{thm4.1}
Under the notations and conventions above, for any $t \in (0,1)$ sufficiently close to 1, there exists 
$\varphi(t, \cdot)= \varphi_t \in \mathcal{H}_K^{4,\alpha} (M)$ such that $\cF_{K}(\varphi_t, t) = 0$ 
and such that $\varphi(1, \cdot)$ is the potential $\varphi_1$ achieving the minimum of $\iota$.
\end{thm}

\begin{proof} The arguments are very similar to the ones used in the proof of Theorem \ref{thm1.1}; for the convenience of the reader, we review here the slight differences.
To start with, the expression of the linearization at $(\varphi_1,1)$ of $\cF_K$ has an additional term, which we now 
compute. 

By differentiating the first term of \eqref{0508}, we obtain
\begin{equation}\label{0509}
\dbar \dot \rho_\varphi(X_1)= \dbar X_1(\dot\varphi)
\end{equation}
so that $\displaystyle \dot \rho_\varphi(X_1)- X_1(\dot\varphi)$ is constant. On the other hand,
by differentiating the second term of \eqref{0508} we infer that we have
\begin{equation}\label{0510}
\int_M \big(\dot\rho_{\varphi}(X_1)+  \rho_{\varphi}(X_1)\Delta_\varphi(\dot\varphi)\big)\omega_{\varphi}^n =0.
\end{equation}
Integration by parts together with the relation \eqref{0508} gives
\begin{equation}\label{0511}
\int_M\dot \rho_\varphi(X_1)- X_1(\dot\varphi) \omega_{\varphi}^n =0
\end{equation}
so in conclusion, we have $\displaystyle \dot \rho_\varphi(X_1)= X_1(\dot\varphi)$. Given the definition of 
$X_1$, this is equivalent to
\begin{equation}\label{0512}
\dot \rho_\varphi(X_1)= \langle \partial \dot\varphi, \bar \partial R_\varphi\rangle_{\omega_\varphi}.
\end{equation}

\noindent Then the derivative of $\cF_K$ at $(\varphi_1, 1)$ has the following expression
\begin{align*}
d\cF|_{(\varphi_1, 1)} : C^{4, \alpha}_K(M) \times \mathbb{R} &\longrightarrow C^{0,\alpha}_K(M)\\
(u , s ) &\longmapsto -\cD_{\varphi_1} u + s (\tr_{\varphi_1} \omega- n).
\end{align*}
where --exactly as in the case of cscK metrics-- the operator $\cD_\varphi$ is the Lichnerowicz operator.

\noindent We define the following functional spaces:
\begin{align*}
\mathcal{H}_{K, \varphi_1}:& =\{ u \in \mathcal{H}_K^{\infty}(M) | \cD_{\varphi_1} u = 0, \int u \omega_{\varphi_1}^n = 0\},\\
\mathcal{H}_{K, \varphi_1, k}^{\perp}:& = \{u \in C^{k,\alpha}_K(M)| \int uv \omega_{\varphi_1}^n = 0 \text{ for any } v \in \mathcal{H}_{K, \varphi_1}, \int u \omega_{\varphi_1}^n = 0\}.
\end{align*}
and then we have the following statement.
\begin{lem}We have the orthogonal decomposition
$$\cC_K^{k, \alpha}(M) =\mathbb{R} \oplus \mathcal{H}_{K,\varphi_1} \oplus \mathcal{H}_{K,\varphi_1, k}^{\perp}.$$
\end{lem}
\begin{proof}  
Indeed, this is a consequence of the fact that the operator $\displaystyle \cD_{\varphi_1}$ is $K$-invariant and self adjoint on $C_K^{\infty}(M)$. And we have $C^{0, \alpha}_K(M)= \cD_{\varphi_1}\big(C^{4, \alpha}_K(M)\big)\oplus \Ker(\cD_{\varphi_1})$
which can be derived from the elliptic operators theory.
\end{proof}

\noindent The rest of the proof of Theorem $\ref{thm4.1}$ is strictly identical to the one presented in the previous 
section, so we will not discuss it further here.
\end{proof}
\medskip

\noindent We prove next Corollary \ref{cor1.4}.

%\begin{cor}\label{cor4.2}
%Suppose there exists two extremal metrics $\omega_{\varphi_1} , \omega_{\varphi_2}\in [\omega]$. Then there exists an element $\sigma \in Aut_0(M)$ such that $\sigma^* \omega_{\varphi_1} = \omega_{\varphi_2}$.
%\end{cor}

\begin{proof} We begin with a few reductions.
By Theorem \ref{aut_cal}, combined with the fact that the maximal compact subgroups of $\Aut_0(M)$ are 
conjugate (by a result of Matsushima), we can assume that we have
$$\Iso(M, \omega_{1})= \Iso(M, \omega_{2}).$$
We can equally assume that $\varphi_j$ is the minimum point of the functional 
$\displaystyle \iota|_{\Psi_{\omega_j}}$, for $j=1, 2$. Then we still have 
\begin{equation}
\Iso(M, \omega_{\varphi_1})= \Iso(M, \omega_{\varphi_2})
\end{equation}
and by Theorem \ref{uniq_vf} we have $X_1= X_2:= X$.
 
Theorem $\ref{thm4.1}$ shows that there exists two paths of twisted extremal metrics, $\varphi_{k, t}$ 
with $\varphi_{k,1} = \varphi_k$ for $k=1,2$ satisfying
$$\nabla_{\varphi_{k, t}}^{1,0} (R_{\varphi_{k, t}} - (1-t)\text{tr}_{\varphi_{k, t}} \chi) = X_k.$$ 
Hence, we get two smooth families ($t \in (1-\epsilon , 1]$) of solutions to the equation 
\begin{align}\label{eqn4.10}
R_{\varphi_t} -\underline{R} - \rho_{\varphi_t}(X)- (1-t)(\text{tr}_{\varphi_t}\omega - n) = 0.
\end{align}
We prove next that for fixed $t \in (0, 1)$, the $K$-invariant smooth solution of $(\ref{eqn4.10})$ is unique. First, we introduce the modified K-energy (c.f. \cite{FM}) on $\mathcal{H}_{K}^{\infty}$
\begin{align*}
\frac{\mathrm{d} \cE_K}{\mathrm{d}t} &= \int_M \big(- (R_{\varphi} - \underline{R}) + \rho_{\varphi} (X)\big)\frac{\mathrm{d} \varphi}{\mathrm{d}t}\omega_{\varphi}^n.
\end{align*}
And $\cE_K$ is weakly convex along any $K$-invariant $C^{1,1}$ geodesic segment by \cite{BB} and \cite{CLP}. Moreover, $\iota$ is strictly convex along $C^{1,1}$ geodesic segments. Therefore, for $t \in (0,1)$ 
\begin{align}\label{eqn4.11}
\cE_K+ (1-t)\iota
\end{align}
is strictly convex along any $K$-invariant $C^{1,1}$ geodesic segment. Also note that any two $K$-invariant K\"ahler potentials can be joined by a $K$-invariant $C^{1,1}$ geodesic. %One can go back to the proof of existence of $C^{1,1}$ geodesic in \cite{C} and carefully construct an sequence of paths of $K$-invariant potentials which converges to a $K$-invariant $C^{1,1}$ geodesic.
% Thus, we can connect any two different critical points of $(\ref{eqn4.11})$ by a $K$-invariant $C^{1,1}$ geodesic. 

By the strict convexity, we can conclude the $K$-invariant solution of $(\ref{eqn4.10})$ is unique. Hence $\varphi_{1,t} = \varphi_{2,t}$ for $t \in (1- \epsilon, 1)$. As $t\rightarrow 1$ we get that $\varphi_1 = \varphi_2$, which is a contradiction, and the proof of Corollary \ref{cor1.4} is finished.
\end{proof}

\end{document}